\documentclass[12pt]{amsart}
\usepackage{amsmath}
\usepackage{stmaryrd}
\usepackage{amssymb}
\usepackage{amsfonts}
\usepackage{amsthm}
\usepackage{xcolor}
\usepackage{enumerate}
\PassOptionsToPackage{hyphens}{url}\usepackage{hyperref}

\hypersetup{
    colorlinks,
    linkcolor={red!50!black},
    citecolor={blue!50!black},
    urlcolor={blue!80!black}
} % To use like \url{www.nkvishnu.wordpress.com}

\newtheorem{theo}{Theorem}[section]
\newtheorem*{theo*}{Theorem}
\newtheorem{coro}[theo]{Corollary}
\newtheorem{lemm}[theo]{Lemma}

\newcommand{\N}{\mathbb{N}}
\newcommand{\Z}{\mathbb{Z}}
\newcommand{\R}{\mathbb{R}}

\title{Asymptotic formulas involving Cohen-Ramanujan expansions}

\begin{document}
\keywords{Jordan totient function; Klee's function; divisor function; Ramanujan Sum; convolution sums; Cohen-Ramanujan Sum; Cohen-Ramanujan Expansions; mean value}
\subjclass[2010]{11A25, 11L03, 11N05, 11N37}
\author[A Chandran]{Arya Chandran}
\address{Department of Mathematics, University College, Thiruvananthapuram, Kerala - 695034, India}
\email{aryavinayachandran@gmail.com}
\author[K V Namboothiri]{K Vishnu Namboothiri}
\address{Department of Mathematics, Government College, Ambalapuzha, Kerala - 688561, INDIA\\Department of Collegiate Education, Government of Kerala, India}
\email{kvnamboothiri@gmail.com}

\begin{abstract}
Some necessary and sufficient conditions for the existence of Cohen-Ramanujan expansions for arithmetical functions were provided by these authors in [\textit{arXive preprint arXive:2205.08466}, 2022]. Given two arithmetical functions $f$ and $g$ with absolutely convergent Cohen-Ramanujan expansions, we derive an asymptotic formula for $\sum_{n\leq N}f(n)g(n+h)$ where $h$ is a fixed positive integer. We also provide Cohen-Ramanujan expansions for certain functions to illustrate some of the results we prove consequently.  
\end{abstract}

 \maketitle
\section{Introduction}
 Srinivasa Ramanujan introduced the following trigonometric sum
\begin{align}
c_r(n)&=\sum\limits_{\substack{{m=1}\\(m,r)=1}}^{r}e^{\frac{2 \pi imn}{r}}\label{ramanujan_sum}
\end{align}
 in \cite{ramanujan1918certain} and used this sum to derive infinite Fourier series like expansions for several arithmetical functions in the form $\sum\limits_{r}a_rc_r(n)$. The trigonometric sum \eqref{ramanujan_sum} now known as the Ramanujan sum has several interesting properties. For example it is multiplicative in $r$, but not so in $n$. To see the other properties in detail, please see \cite{tom1976introduction},  \cite{mccarthy2012introduction} or \cite{sivaramakrishnan2018classical}. When such a Fourier series like expansion exists for an arithmetical function (mostly only pointwise convergent), the  function is said to possess a Ramanujan-Fourier series expansion or simply a Ramanujan expansion. Ramanujan expansions were computed for various functions by many authors. Various conditions were provided for the existence of Ramanujan expansions in papers like \cite{hardy1921note},  \cite{lucht2010survey} and \cite{murty2013ramanujan}.

Ramanujan sum has been generalized in many directions. For example, in \cite{cohen1949extension}, E. Cohen defined the sum
\begin{align}\label{gen-ram-sum}
c_r^s(n)&=\sum\limits_{\substack{h=1\\{(h,r^s)_s=1}}}^{r^s}e^{\frac{2\pi i n h}{r^s}},
\end{align}
 generalizing the Ramanujan sum. For a positive integer $s$, integers $m,n$, not both zero, the generalized GCD of $m$ and $n$ denoted by $(m,n)_{s}$ is the largest $l^s$ (where $l\in \N$) dividing both $m$ and $n$. If $(m,n)_s=1$, $m$ and $n$ are said to be relatively $s$-prime to each other. When $s=1$, generalized GCD becomes the usual GCD and so the sum (\ref{gen-ram-sum}) reduces to the usual Ramanujan sum in that case.
 We will refer to the sum \eqref{gen-ram-sum} by the name Cohen-Ramanujan sum. Just like the usual Ramanujan sum, for certain specific values of $s$ and $n$, this sum transforms into various well known arithmetical functions. For example, $c_r^s(0) = \Phi_s(r^s)=J_s(r)$, where $\Phi_s$ is the Klee's function and $J_s$ is the Jordan totient function. Please see the next section for the definitions of these two functions.  Cohen derived various properties of this sum and proved several results using this generalization in a series of papers \cite{cohen1949extension,cohen1955extensionofr,cohen1956extension}.
 
 E. Cohen \cite{cohen1959trigonometric} himself gave another generalization of the Ramanujan sum using $k$-vectors. A $k$-vector is  an ordered set $\{x_i\}=\{x_1,\ldots,x_k\}$ of $k$ integers. Two $k$-vectors $\{x_i\}$ and $\{y_i\}$ are congruent (mod $k,r$) if  $x_i\equiv y_i$ (mod $r$), $i=1,\ldots,k$. For $k\geq 1$, Cohen defined
\begin{align}\label{coh_gen2}
c^k(n,r) &= \sum\limits_{(\{x_i\},r)=1}e^{\frac{2\pi i n (x_1+x_2+\ldots x_k)}{r}},
\end{align} where $\{x_i\}$ ranges over a residue system (mod $k,r$). 
 
 Based on this sum, Sivaramakrishnan in \cite{sivaramakrishnan1988classical} gave the infinite series expansion
\begin{align}\label{Jordan_sum}
\frac{J_s(n^s)}{n^s}\zeta(s+k) = \sum\limits_{\substack{r=1}}^{\infty}\frac{\mu(r)c^k(n,r)}{J_{s+k}(r)},
\end{align}
where $s$ and $k$  are positive integers with $s+k \geq 2$.
 
 Another generalization of the Ramanujan sum, and a list of references to some more generalizations can be found in \cite{subba1966new}. Later, more generalizations were derived in \cite{ramanathan1980some}. Recently, Haukkanen and McCarthy gave another generalization in \cite{haukkanen1991sums}.

 An arithmetical function $f$ is said to admit a Cohen-Ramanujan expansion 
 \begin{align*}
 f(n)=\sum\limits_{r}\widehat{f}(r)c_r^s(n),
 \end{align*}
if the series on the right hand side converges for suitable complex numbers $\widehat{f}(r)$. A study on the Cohen-Ramanujan expansions of some arithmetical functions were conducted by these authors in \cite{Chandran2022OnAR}. Some  necessary and sufficient conditions for the existence of such expansions were also given there. To the best of our knowledge, such series expansions are not available for any other generalization of the Ramanujan sum as of now.

 In this paper, we use the Cohen-Ramanujan expansion in the context initiated by the attempts of H. Gadiyar and R. Padma in \cite{gopalakrishna2014ramanujanczech}. Hardy derived a Ramanujan expansion \cite[Identity 2.23]{gopalakrishna2014ramanujanczech} for the function $\Lambda_1(n) = (\phi(n)/n)\Lambda(n)$ where $\Lambda(n)$ is the von Mangoldt function defined as 
 \begin{align}
  \Lambda(n) = \begin{cases}
                \log p,\text{ if $n = p^k$ , $p$ is prime and $k$ any positive integer}\\
                0, \text{ otherwise}.
               \end{cases}
 \end{align}
 
 If $a$ is an arithmetical function with the Ramanujan expansion \begin{align}
                                                                  a(n) = \sum\limits_{r = 1}^{\infty}a_rc_r(n)
                                                                 \end{align}
then the Wiener–Khintchine theorem for an arithmetical function $a$ \cite[Section 3.2]{gopalakrishna2014ramanujanczech} can be stated as \begin{align}                                                                                                                                
\lim\limits_{n \longrightarrow \infty }\frac{1}{N}\sum\limits_{n\le q N} a(n) a(n+h)= \sum\limits_{\substack{r=1}}^{\infty}a_r^2 c_r(h)
\end{align} provided that certain convergence conditions are satisfied.

H. Gadiyar and R. Padma showed that Hardy-Littlewood conjecture on twin primes can be proved if the Wiener–Khintchine theorem can be applied to the Ramanujan expansion of the arithmetical function $\Lambda_1$ provided that the expansion is uniformly and absolutely convergent. They derived the formula  
\begin{align*}
\sum\limits_{n \leq N} \Lambda(n) \Lambda(n+h) \sim N \sum\limits_{\substack{r = 1}}^{\infty}\frac{\mu^2(r)}{\phi(r)}c_r(h),
\end{align*} which agrees with the Hardy-Littlewood conjecture. But note that this identity is valid only if the Ramanujan expansion derived by Hardy is absolutely convergent, which is not actually the case.   Taking forward these discussions, H. Gadiyar, M. Ram Murthy and R. Padma in \cite{gopalakrishna2014ramanujan} derived an asymptotic formula for the sum
$\sum\limits_{n\leq N} f(n)g(n+h),$
where $f$ and $g$ are two arithmetical functions with absolutely convergent Ramanujan expansions.

Here, we try to derive an asymptotic formula for the sum $\sum\limits_{n\leq N} f(n)g(n+h)$,
where $f$ and $g$ are arithmetical functions with absolutely convergent Cohen-Ramanujan expansions and satisfy some additional conditions. Therefore, our results extend the results in \cite{gopalakrishna2014ramanujan}. We also provide some examples to demonstrate our results. The main results we propose in this paper are the following.

\begin{theo}\label{theo1}
Suppose that $f$ and $g$ are two arithmetical functions with absolutely convergent Cohen-Ramanujan expansions $f(n) = \sum\limits_{\substack{r}}\widehat{f}(r)c_r^{(s)}(n)$ and  $g(n) = \sum\limits_{\substack{k}}\widehat{g}(k)c_k^{(s)}(n)$ respectively. Suppose that $\sum\limits_{\substack{r,k}}\vert \widehat{f}(k)\widehat{g}(k) \vert (r^s,k^s)_s \tau_s(r^s) \tau_s(k^s)< \infty$. Then as $N$ tends to infinity, $\sum\limits_{\substack{n \leq N}} f(n)g(n)\sim N \sum\limits_{\substack{r}} \widehat{f}(r)\widehat{g}(r) \Phi_s(r^s)$.
\end{theo} 

\begin{theo}\label{con-sum}
Suppose that $f$ and $g$ are two arithmetical functions with absolutely convergent Cohen-Ramanujan expansions $f(n) = \sum\limits_{\substack{r}}\widehat{f}(r)c_r^{(s)}(n)$ and $g(n) = \sum\limits_{\substack{k}}\widehat{g}(k)c_k^{(s)}(n)$ respectively. Suppose that $\sum\limits_{\substack{r,k}}\vert \widehat{f}(r)\vert \vert\widehat{g}(k) \vert (r^sk^s)^{\frac{1}{2}} \tau_s(r^s) \tau_s(k^s)< \infty$. Then as $N$ tends to infinity, $\sum\limits_{\substack{n \leq N}} f(n)g(n+h)\sim N \sum\limits_{\substack{r}} \widehat{f}(r)\widehat{g}(r) c_r^s(h)$.

\end{theo}

Given an arithmetical function $f$, if we write $f_h(n) = f(n+h)$, we answer a natural question that if $f$ has a Cohen-Ramanujan expansion, then does $f_h$ also have one?

\begin{theo}\label{th:fhexpansion}
Suppose that $f$ has an absolutely convergent Cohen-Ramanujan expansion with coefficients $\widehat{f}(q)$ satisfying $ \sum\limits_{\substack{q = 1}}^{\infty} \vert \widehat{f}(q) \vert \tau(q) < \infty$. Then $f_h$ has an absolutely convergent Cohen-Ramanujan expansion with coefficients $\frac{\widehat{f}(r)c_r^s(h)}{\Phi_s(r^s)}$.
\end{theo}

\section{Notations and basic results}
 Most of the notations, functions and identities we use in this paper are standard and their definitions can be found in \cite{tom1976introduction} or \cite{mccarthy2012introduction}. However, for the sake of completeness, we define the most important ones below. 
As usual $[m, n]$ denotes the lcm of integers $m,n$. As in the case of the well known identity $(m,n)[m,n]=mn$ it is easy to verify that $ (m^s,n^s)_s [m^s,n^s]=m^sn^s$.
 
  The Jordan totient function $J_s(n)$ defined for positive integers $s$ and $n$ gives the number of ordered sets of $s$ elements from a complete residue system (mod $n$) such that the greatest common divisor of each set is prime to $n$.
 The Klee's function $\Phi_{s}$  is defined 
   	to give the cardinality of the  set $\{m\in\N : 1\leq m\leq n, (m,n)_s=1\}$. 
  Note that $\Phi_1 = \varphi$, the Euler totient function. It is known that \cite[Section V.3]{sivaramakrishnan1988classical}  $n^s =  \sum\limits_{d|n}J_s(d)$ and $J_s(n) = \Phi_s(n^s)$. Hence 
  \begin{align}\label{phi-reltn}
  n^s &=  \sum\limits_{d|n}\Phi_s(d^s) =\sum\limits_{d^s|n^s}\Phi_s(d^s).
  \end{align}
 By $\tau_{s}(n)$ where $s,n\in \N$, we mean the number of $l^s$ dividing $n$  with $l\in \N$. $\zeta(s)$ denotes the Riemann zeta function which is defined by the equation $\zeta(s) = \sum\limits_{\substack{n=1}}^{\infty} \frac{1}{n^s}$ for $\Re(s) >1$. 
 
\begin{lemm}\label{l1}\cite[Theorem 3.2]{tom1976introduction}
If $x \geq 1$ we have, $\sum\limits_{\substack{n \leq x}}\frac{1}{n} = \log x+C+O(\frac{1}{x})$, where $C$ is the Euler constant.
\end{lemm}

\begin{lemm}\label{prod_zeta}\cite[Lemma 1.2]{ben2000zeta}
$\zeta(z) = \prod\limits_{p}\frac{1}{1-p^{-z}}$. Further, $\zeta(z)$ converges for all $z$ with $\Re(z) \geq 1$.
\end{lemm}

Let $x$ be a real number. We denote the distance from $x$ to the nearest integer by $\Vert x \Vert$.

The following lemma is essential to prove one of the main results that we propose in this paper.
 	
 \begin{lemm}\label{l2}\cite[Lemma 1]{korobov2013exponential} 
For any $\alpha \in \R, p,q \in \Z$( with $p>0$), we have $\vert \sum\limits_{\substack{k=q+1}}^{q+p}e^{2\pi i \alpha k} \vert \leq min\{p, \frac{1}{2\Vert \alpha\Vert}\}$.
\end{lemm}

For an arithmetical function $f$, its mean value is defined by $M(f) = \lim\limits_{\substack{x\rightarrow \infty}}\frac{1}{x} \sum\limits_{\substack{{n \leq x}}} f(n)$, when the limit exists. 
For $x \in \R$,  $[[x]]$ denotes  the largest integer $n$ such that $n \leq x$.
We prove a lemma using elementary techniques.
\begin{lemm} \label{l3}

For $n, N \in \N$ and $h >0$, $\sum\limits_{\substack{n \leq N \\ d \mid n \\e \mid n}}1=\llbracket \frac{N}{[d,e]}\rrbracket$
and $\sum\limits_{\substack{n \leq N \\ d \mid n+h \\e \mid n+h}}1\leq \llbracket \frac{N+ h}{[d,e]}\rrbracket$.
\end{lemm}
\begin{proof} The first sum is
\begin{align*}
 \sum\limits_{\substack{n \leq N \\ d \mid n \\e \mid n}}1 &= \sum\limits_{\substack{n \leq N \\ [d,e] \mid n}}1
 \\& = \sum\limits_{\substack{n \leq N \\ k \mid n}}1, \text{ where } k=[d,e]
 \\& = \llbracket \frac{N}{k}\rrbracket
 \\&=  \llbracket \frac{N}{[d,e]}\rrbracket
\end{align*}
Now $\sum\limits_{\substack{n \leq N \\ d \mid n+h \\e \mid n+h}}1= \sum\limits_{\substack{n \leq N \\ [d,e] \mid n+h }}1=\sum\limits_{\substack{n \leq N \\ k \mid n+h }}1$, where $k=[d,e]$. Here $k$ runs through the divisors $k,2k,3k,\ldots,mk$ where $mk\leq N+h$. Therefore number of such $k$'s $=m\leq \frac{N+h}{k}$. Since $m \in \N$, $m \leq \llbracket \frac{N+h}{k}\rrbracket$.
\end{proof}
\begin{lemm}
For $n \in N$, the number of divisors $d$ of $n$ with a fixed $\delta$ as a factor of $d$ is $\tau(\frac{n}{\delta})$.
\end{lemm}
\begin{proof}
We have to find the number of $d$ such that $d \mid n$ and $\delta \mid d$. That is $d=\delta t$ for some integer $t$ and $d \mid n$. Which implies, $n= \delta t k$, for some integer $k$. That is 
$\frac{n}{\delta} = t k$.
Thus we get number of such $d's =$ number of $t's$ such that $\frac{n}{\delta} = tk$, which is exactly $\tau(\frac{n}{\delta}).$
\end{proof}

 \section{Proofs of the Main Results}
We start by proving three easy lemmas essential in the proof of our main results.
\begin{lemm}\label{lem1}
$\sum\limits_{\substack{n \leq N}} c_r^s(n)c_k^s(n) \leq N \tau_s(r^s)\tau_s(k^s)(r^s,k^s)_s$.
\end{lemm} 
\begin{proof}
\begin{align*}
\sum\limits_{\substack{n \leq N}} c_r^s(n)c_k^s(n) &
= \sum\limits_{\substack{n \leq N}} \sum\limits_{\substack{d \mid r\\d^s\mid n}}\mu(\frac{r}{d})d^s \sum\limits_{\substack{e \mid k\\e^s\mid n}}\mu(\frac{k}{e})e^s
\\&=\sum\limits_{\substack{d \mid r\\e\mid k}}\mu(\frac{r}{d})\mu(\frac{k}{e})d^s e^s \sum\limits_{\substack{n \leq N\\d^s \mid k\\e^s\mid n}}1
\\&= \sum\limits_{\substack{d \mid r\\e\mid k}}\mu(\frac{r}{d})\mu(\frac{k}{e})d^s e^s \llbracket \frac{N}{[d^s,e^s]}\rrbracket \text{\hspace{1cm}(by Lemma \ref{l3}})
\\& \leq \sum\limits_{\substack{d \mid r\\e\mid k}}d^s e^s \frac{N}{[d^s,e^s]}
\\& \leq N \sum\limits_{\substack{d \mid r\\e\mid k}}(d^s,e^s)_s \hspace{.2 cm}(\text{since } (d^s,e^s)_s[d^s,e^s] = d^s e^s)
\\&  = N \sum\limits_{\substack{d \mid r\\e\mid k}}\sum\limits_{\substack{\delta^s \mid (d^s,e^s)_s}}\Phi_s(\delta^s)\hspace{1cm} \text{(by equation }(\ref{phi-reltn}))
\\&  = N \sum\limits_{\substack{d^s \mid r^s\\e^s\mid k^s}}\sum\limits_{\substack{\delta^s \mid d^s\\\delta^s \mid e^s}}\Phi_s(\delta^s).
\end{align*}
When we take the sum $\sum\limits_{\substack{d^s \mid r^s}}\sum\limits_{\substack{\delta^s \mid d^s}}\Phi_s(\delta^s)$, each $\delta$ is repeated $\tau(\frac{r^s}{\delta^s})$ times.
Now
\begin{align*}
\sum\limits_{\substack{n \leq N}} c_r^s(n)c_k^s(n)
 &\leq    N \sum\limits_{\substack{d^s \mid r^s\\e^s\mid k^s}}\sum\limits_{\substack{\delta^s \mid (d^s,e^s)_s}}\Phi_s(\delta^s) 
 \\&\leq N \sum\limits_{\substack{\delta^s \mid (r^s,k^s)_s}}\Phi_s(\delta^s) \tau_s(\frac{r^s}{\delta^s})\tau_s(\frac{k^s}{\delta^s})
\\& \leq N \sum\limits_{\substack{\delta^s \mid (r^s,k^s)_s}}\Phi_s(\delta^s) \tau_s(r^s)\tau_s(k^s)
\\& = N \tau_s(r^s)\tau_s(k^s) \sum\limits_{\substack{\delta^s \mid (r^s,k^s)_s}}\Phi_s(\delta^s)
\\& = N \tau_s(r^s)\tau_s(k^s)(r^s,k^s)_s \hspace{1cm} \text{(by equation }(\ref{phi-reltn})).
\end{align*}
\end{proof}

\begin{lemm}\label{lem2}
$\sum\limits_{\substack{n \leq N}} c_r^s(n)c_k^s(n+h) =\delta_{rk} N c_r^{s}(h)+O(r^sk^s \log r^sk^s)$, where $\delta_{rk}$ is the Kronecker delta. 
\end{lemm}
\begin{proof}
\begin{align*}
\sum\limits_{\substack{n \leq N}} c_r^s(n)c_k^s(n+h)
& = \sum\limits_{\substack{n \leq N}} \sum\limits_{\substack{a=1\\(a,r^s)_s=1}}^{r^s} e^{\frac{2\pi i a n}{r^s}}\sum\limits_{\substack{b=1\\(b,k^s)_s=1}}^{k^s} e^{\frac{2\pi i b( n+h)}{k^s}}
\\&= \sum\limits_{\substack{n \leq N}} \sum\limits_{\substack{a=1\\(a,r^s)_s=1}}^{r^s}\sum\limits_{\substack{b=1\\(b,k^s)_s=1}}^{k^s}e^{\frac{2\pi i bh}{k^s}}e^{2\pi i (\frac{a}{r^s}+\frac{b}{k^s})n}
\\&= \sum\limits_{\substack{n \leq N}} \sum\limits_{\substack{a=1\\(a,r^s)_s=1}}^{r^s}\sum\limits_{\substack{b=1\\(b,k^s)_s=1}}^{k^s}e^{\frac{2\pi i bh}{k^s}}e^{2\pi i (\frac{ak^s+br^s}{r^sk^s})n}.
\end{align*}
Now we have two cases.\\
Case 1 : $\frac{a}{r^s}+\frac{b}{k^s} \in \mathbb{Z}$.\\
In this case 
\begin{align*}
\frac{ak^s+br^s}{r^sk^s} \in \mathbb{Z}
&\Longleftrightarrow r^sk^s \mid ak^s+br^s
\\&\Longleftrightarrow ak^s+br^s=tr^sk^s,\text{ for some } t \in \mathbb{Z}
\\&\Longleftrightarrow br^s= tr^sk^s-ak^s.
\end{align*}
Since $k^s \mid RHS $ and $(b,k^s)_s=1$, $k \mid r$. Using similar arguments, we get $r \mid k$, which implies $r=k$. Now we have  $r^s\mid a+b$. Since $1 \leq a<r^s$ and $1 \leq b<r^s$, $r^s=a+b$.
Then $ \sum\limits_{\substack{n \leq N}} e^{2\pi i (\frac{ak^s+br^s}{r^sk^s})n}=  \sum\limits_{\substack{n \leq N}}1 = N.$ Therefore,
\begin{align*}
\sum\limits_{\substack{n \leq N}} c_r^s(n)c_k^s(n+h)
&=N c_r^s(h).
\end{align*}
Case 2 :  $\frac{a}{r^s}+\frac{b}{k^s} \notin \mathbb{Z}$.\\
In this case $e^{2\pi i (\frac{ak^s+br^s}{r^sk^s})n} \neq 1$. Let us write $N = (r^sk^s)Q+R$, where $0\leq R <r^sk^s$. Therefore,
\begin{align*}
\vert\sum\limits_{\substack{n \leq N}} e^{2\pi i (\frac{ak^s+br^s}{r^sk^s})n}\vert 
&= \vert \sum\limits_{\substack{n \leq (r^sk^s)Q}} e^{2\pi i (\frac{ak^s+br^s}{r^sk^s})n} + \sum\limits_{\substack{(r^sk^s)Q <n \leq N}} e^{2\pi i (\frac{ak^s+br^s}{r^sk^s})n}\vert 
\\&= \vert 0+ \sum\limits_{\substack{(r^sk^s)Q <n \leq N}} e^{2\pi i (\frac{ak^s+br^s}{r^sk^s})n} \vert.
\end{align*}
Now 
\begin{align*}
\vert \sum\limits_{\substack{n \leq N}} e^{2\pi i (\frac{ak^s+br^s}{r^sk^s})n} \vert
&= \vert\sum\limits_{\substack{(r^sk^s)Q <n \leq N}} e^{2\pi i (\frac{ak^s+br^s}{r^sk^s})n} \vert
\\& \leq \frac{1}{\Vert \frac{ak^s+br^s}{r^sk^s}\Vert}  \text{\hspace{1cm}(by Lemma \ref{l2}}). 
\end{align*} This is because, $e^{\frac{2 \pi i kl}{r^sk^s}}$ is a proper non-trivial root of unity, $\sum\limits_{\substack{k=1}}^{r^sk^s} e^{\frac{2 \pi i kl}{r^sk^s}}=0$ if $l \neq 0$.
\begin{align*}
\vert \sum\limits_{\substack{n \leq N}} \sum\limits_{\substack{a=1\\(a,r^s)_s}=1}^{r^s}\sum\limits_{\substack{b=1\\(b,k^s)_s}=1}^{k^s}e^{\frac{2\pi i bh}{ks}}e^{2\pi i (\frac{ak^s+br^s}{r^sk^s})n} \vert
&\leq \sum\limits_{\substack{a=1\\(a,r^s)_s}=1}^{r^s}\sum\limits_{\substack{b=1\\(b,k^s)_s}=1}^{k^s} \frac{1}{\Vert \frac{ak^s+br^s}{r^sk^s}\Vert}.
\end{align*}
 Since $1 \leq a \leq r^s, (a,r^s)_s=1$ and  $1 \leq b \leq k^s, (b,k^s)_s=1$, there are $\Phi_s(r^s)$ choices for $a$ and $\Phi_s(k^s)$ choices for $b$. Also $(r^s,k^s)_s =1$.  Therfore there are $\Phi_s(r^s)\Phi_s(k^s)=\Phi_s(r^sk^s)$ choices for $ak^s+br^s$. Now we will prove that all these are distinct. Suppose that
 $a_1k^s+b_1r^s = a_2k^s+b_2r^s$. Then $(a_1-a_2)k^s = (b_2-b_1)r^s  \Rightarrow r^s \mid a_1-a_2$.
   Since $1 \leq a \leq r^s, a_1=a_2$. Simillarly we get $b_1=b_2$.

 Therefore $ak^s+br^s$ runs through an $s-$reduced residue system modulo $r^sk^s$. Thus
\begin{align*}
\sum\limits_{\substack{a=1\\(a,r^s)_s=1}}^{r^s}\sum\limits_{\substack{b=1\\(b,k^s)_s=1}}^{k^s} \frac{1}{\Vert \frac{ak^s+br^s}{r^sk^s}\Vert} 
& \leq \sum\limits_{\substack{ak^s+br^s=1\\(ak^s+br^s,r^sk^s)_s=1}}^{r^sk^s} \frac{1}{\Vert \frac{ak^s+br^s}{r^sk^s}\Vert}
\\& = \sum\limits_{\substack{t=1\\(t,r^sk^s)_s=1}}^{r^sk^s} \frac{1}{\Vert \frac{t}{r^sk^s}\Vert}
\\& \leq \sum\limits_{\substack{t=1\\(t,r^sk^s)_s=1}}^{r^sk^s} \frac{1}{ \frac{t}{r^sk^s}}
\\&  =r^sk^s \sum\limits_{\substack{t=1\\(t,r^sk^s)_s=1}}^{r^sk^s}\frac{1}{t}
\\& \leq r^sk^s \sum\limits_{\substack{t \leq r^sks}}\frac{1}{t}
\\&  =r^sk^s(\log(r^s k^s)+C+O(\frac{1}{r^sk^s}))  \text{\hspace{.2cm}(by Lemma \ref{l1}})
\\&  << r^sk^s\log(r^s k^s).
\end{align*}
When $(r^s,k^s)_s=d^s >1$, we may write $r^s=d^sr_1^s$ and $k^s=d^sk_1^s$ so that $(r_1^s,k_1^s)_s=1$.
\begin{align*}
\frac{a}{r^s}+\frac{b}{k^s} &= \frac{a}{r_1^sd^s}+\frac{b}{k_1^sd^s} \text{ for some }r_1,k_1 \in \mathbb{Z}
\\&= \frac{ak_1^s+br_1^s}{r_1^sk_1^sd^s}
\\&= \frac{ak_1^s+br_1^s}{d^s \frac{r^s}{d^s}\frac{k^s}{d^s}}
\\&= \frac{ak_1^s+br_1^s}{\frac{r^sk^s}{d^s}}
\\&= \frac{ak_1^s+br_1^s}{\frac{r^sk^s}{(r^s,k^s)_s}}
\\&= \frac{ak_1^s+br_1^s}{[r^s,k^s]}.
\end{align*}
Therefore, $\Vert \frac{ak^s+br^s}{r^sk^s}\Vert = \Vert \frac{ak^s+br^s}{[r^s,k^s]}\Vert$. When $ 1 \leq a \leq r^s$ with $(a,r^s)_s=1$ and  $ 1 \leq b \leq k^s$ with $(b,k^s)_s=1$, the expression $ak_1^s+br_1^s$ ranges over residue classes modulo $[r^s,k^s]$. Since $(r_1^s,k_1^s)=1,$ these residue classes are distinct mod $(r_1^s,k_1^s)$. Therefore each class modulo $[r^s,k^s]$ is repeated atmost $d^s$ times. Then we get,
\begin{align*}
\vert \sum\limits_{\substack{n \leq N}} \sum\limits_{\substack{a=1\\(a,r^s)_s=1}}^{r^s}&\sum\limits_{\substack{b=1\\(b,k^s)_s=1}}^{k^s}e^{\frac{2\pi i bh}{ks}}e^{2\pi i (\frac{ak^s+br^s}{r^sk^s})n} \vert
\\&\leq \sum\limits_{\substack{a=1\\(a,r^s)_s=1}}^{r^s}\sum\limits_{\substack{b=1\\(b,k^s)_s=1}}^{k^s} \frac{1}{\Vert \frac{ak^s+br^s}{r^sk^s}\Vert}
\\& \leq \sum\limits_{\substack{a=1\\(a,r^s)_s}=1}^{r^s}\sum\limits_{\substack{b=1\\(b,k^s)_s}=1}^{k^s} \frac{1}{\Vert\frac{ak_1^s+br_1^s}{[r^s,k^s]}\Vert}
\\& \leq \sum\limits_{\substack{ak^s+br^s=1}}^{r^sk^s} \frac{1}{\Vert\frac{ak_1^s+br_1^s}{[r^s,k^s]}\Vert}
\\&  = \sum\limits_{\substack{ak^s+br^s=1}}^{r^sk^s} \frac{1}{\Vert\frac{ak_1^s+br_1^s}{d^sr_1^sk_1^s}\Vert}\text{ \big(since }(r^s,k^s)_s[r^s,k^s]=r^sk^s,\\& \hspace{3cm}
r^s=d^sr_1^s, k^s = d^s k_1^s \text{ and } d^s = (r^s,k^s)_s\big)
\\& =  \sum\limits_{\substack{t=1}}^{r^sk^s} \frac{1}{\Vert\frac{t}{d^sr_1^sk_1^s}\Vert}
\\& \leq \sum\limits_{\substack{t=1}}^{r^sk^s} \frac{1}{\frac{t}{d^sr_1^sk_1^s}}
\\& \leq d^s \big( \sum\limits_{\substack{t =1}}^{[r^s,k^s]}\frac{d^sr_1^sk_1^s}{t}+ \cdots + \sum\limits_{\substack{t =1}}^{[r^s,k^s]}\frac{d^sr_1^sk_1^s}{t}\big)
\\& \leq d^s d^sr_1^sk_1^s \big( \sum\limits_{\substack{t =1}}^{[r^s,k^s]}\frac{1}{t}+ \cdots + \sum\limits_{\substack{t =1}}^{[r^s,k^s]}\frac{1}{t}\big)
\\& \leq d^s [r^s,k^s] \big( d^s\sum\limits_{\substack{t =1}}^{[r^s,k^s]}\frac{1}{t}\big)
\\& \leq d^s [r^s,k^s] \big( d^s\sum\limits_{\substack{t \leq [r^s,k^s]}}\frac{1}{t}\big)
\\& \leq d^s [r^s,k^s] \big( d^s(\log [r^s,k^s]+C+O(\frac{1}{[r^s,k^s]}))\big) \text{\hspace{.2cm}(by Lemma \ref{l1}})
\\& \leq d^s [r^s,k^s] \big(d^s\log [r^s,k^s]+ d^sC+O(\frac{ d^s}{[r^s,k^s]}))\big)
\\& << d^s [r^s,k^s] \log d^s[r^s,k^s]
\\&  = r^sk^s \log(r^sk^s). 
\end{align*}
\end{proof}
\begin{lemm}\label{lem3}
$\vert \sum\limits_{\substack{n \leq N}} c_r^s(n)c_k^s(n+h)\vert \leq  N^{\frac{1}{2}} (N+ h )^{\frac{1}{2}} (r^sk^s)^{\frac{1}{2}}  \tau_s(r^s)\tau_s(k^s)$, where $h>0$.
\end{lemm}
\begin{proof}
By Cauchy-Schwartz inequality,
\begin{align*}
{\vert \sum\limits_{\substack{n \leq N}} c_r^s(n)c_k^s(n+h)\vert }^2
& \leq \sum\limits_{\substack{n \leq N}} {\vert c_r^s(n) \vert}^2\sum\limits_{\substack{n \leq N}} {\vert c_k^s(n+h) \vert}^2.
\end{align*}
Now consider the sum
\begin{align*}
\sum\limits_{\substack{n \leq N}} {\vert c_r^s(n) \vert}^2 & \leq \sum\limits_{\substack{n \leq N}} c_r^s(n)c_r^s(n)
\\& = \sum\limits_{\substack{n \leq N}} \sum\limits_{\substack{d \mid r\\d^s\mid n}}\mu(\frac{r}{d})d^s  \sum\limits_{\substack{e \mid r\\e^s\mid n}}\mu(\frac{r}{e})e^s
\\& = \sum\limits_{\substack{d \mid r\\e \mid r}}\mu(\frac{r}{d})\mu(\frac{r}{e})d^se^s \sum\limits_{\substack{n \leq N\\d^s\mid n\\e^s\mid n}}1
\\& = \sum\limits_{\substack{d \mid r\\e \mid r}}\mu(\frac{r}{d})\mu(\frac{r}{e})d^se^s \llbracket \frac{N}{{[d^s,e^s]}}\rrbracket
\\& \leq \sum\limits_{\substack{d \mid r\\e \mid r}}d^se^s \frac{N}{{[d^s,e^s]}}
\\& \leq N \sum\limits_{\substack{d \mid r\\e \mid r}}(d^s,e^s)_s
\\& = N \sum\limits_{\substack{d \mid r\\e \mid r}}\sum\limits_{\substack{\delta^s \mid (d^s,e^s)_s}}\Phi_s(\delta^s)\hspace{2cm} \text{(by equation  }(\ref{phi-reltn}))
\\& = N \sum\limits_{\substack{d^s \mid r^s\\e^s \mid r^s}}\sum\limits_{\substack{\delta^s \mid (d^s,e^s)_s}}\Phi_s(\delta^s)
\\& =N \sum\limits_{\substack{\delta^s \mid r^s}}\Phi_s(\delta^s)\tau_s(\frac{r^s}{\delta^s})\tau_s(\frac{r^s}{\delta^s})
\\& =N \sum\limits_{\substack{\delta^s \mid r^s}}\Phi_s(\delta^s)(\tau_s(\frac{r^s}{\delta^s}))^2
\\&  \leq N \sum\limits_{\substack{\delta^s \mid r^s}}\Phi_s(\delta^s)(\tau_s(r^s))^2
\\&  = N r^s(\tau_s(r^s))^2.
\end{align*}
Now
\begin{align*}
\sum\limits_{\substack{n \leq N}} {\vert c_k^s(n+h) \vert}^2 & \leq \sum\limits_{\substack{n \leq N}} c_k^s(n+h)c_k^s(n+h)
\\& = \sum\limits_{\substack{n \leq N}} \sum\limits_{\substack{d \mid k\\d^s\mid n+h}}\mu(\frac{k}{d})d^s  \sum\limits_{\substack{e \mid k\\e^s\mid n+h}}\mu(\frac{k}{e})e^s
\\& = \sum\limits_{\substack{d \mid k\\e \mid k}}\mu(\frac{k}{d})\mu(\frac{k}{e})d^se^s \sum\limits_{\substack{n \leq N\\d^s\mid n+h\\e^s\mid n+h}}1
\\& \leq \sum\limits_{\substack{d \mid k\\e \mid k}}\mu(\frac{k}{d})\mu(\frac{k}{e})d^se^s \llbracket \frac{N+h}{{[d^s,e^s]}}\rrbracket
\\& \leq \sum\limits_{\substack{d \mid k\\e \mid k}}\mu(\frac{k}{d})\mu(\frac{k}{e})d^se^s \frac{N+h}{{[d^s,e^s]}}
\\& \leq \sum\limits_{\substack{d \mid k\\e \mid k}}\mu(\frac{k}{d})\mu(\frac{k}{e})(d^s,e^s)_s (N+h)
\\& \leq (N+h) \sum\limits_{\substack{d \mid k\\e \mid k}}(d^s,e^s)_s
\\& = (N+h) \sum\limits_{\substack{d \mid k\\e \mid k}} \sum\limits_{\substack{\delta^s \mid (d^s,e^s)_s}}\Phi_s(\delta^s)\hspace{1cm}\text{(by equation  }(\ref{phi-reltn}))
\\& =(N+h) \sum\limits_{\substack{d^s \mid k^s\\e^s \mid k^s}} \sum\limits_{\substack{\delta^s \mid (d^s,e^s)_s}}\Phi_s(\delta^s)
\\& =(N+h) \sum\limits_{\substack{\delta^s \mid k^s}}\Phi_s(\delta^s)\tau_s(\frac{k^s}{\delta^s})\tau_s(\frac{k^s}{\delta^s})
\\& =(N+h) \sum\limits_{\substack{\delta^s \mid k^s}}\Phi_s(\delta^s)(\tau_s(\frac{k^s}{\delta^s}))^2
\\&  \leq (N+h) \sum\limits_{\substack{\delta^s \mid k^s}}\Phi_s(\delta^s)(\tau_s(k^s))^2
\\& = (N+h) k^s(\tau_s(k^s))^2.
\end{align*}
Therefore 
\begin{align*}
{\vert \sum\limits_{\substack{n \leq N}} c_r^s(n)c_k^s(n+h)\vert }^2& \leq N r^s \tau_s(r^s)^2 (N+h)k^s \tau_s(k^s)^2.\\ 
{\vert \sum\limits_{\substack{n \leq N}} c_r^s(n)c_k^s(n+h)\vert }& \leq N^{\frac{1}{2}} (N+h)^{\frac{1}{2}} (r^sk^s)^{\frac{1}{2}} \tau_s(r^s) \tau_s(k^s).
\end{align*} 
\end{proof}
Now we prove the first result we stated in the introduction. 
\begin{proof}[Proof of Theorem \ref{theo1}]
By the assumption on $f$ and $g$, we have
\begin{align*}
 \sum\limits_{\substack{n \leq N}} f(n)g(n) & =  \sum\limits_{\substack{n \leq N}}  \sum\limits_{\substack{r}}\widehat{f}(r)c_r^{(s)}(n) \sum\limits_{\substack{k}}\widehat{g}(k)c_k^{(s)}(n)\\
 &= \sum\limits_{\substack{r,k}}\widehat{f}(r) \widehat{g}(k)\sum\limits_{\substack{n \leq N}}c_r^{(s)}(n)c_k^{(s)}(n).
 \end{align*}
 Now we split the outersum over $r,k$ into two parts. The first part is over $rk\leq U$ with $U$ to be chosen later tending to infinity. The second part over $rk>U$.
 \begin{align}\label{equ-4}
 \sum\limits_{\substack{n \leq N}} f(n)g(n)&= \sum\limits_{\substack{rk \leq U}}\widehat{f}(r) \widehat{g}(k)\sum\limits_{\substack{n \leq N}}c_r^{(s)}(n)c_k^{(s)}(n)\nonumber \\&+\sum\limits_{\substack{rk> U}}\widehat{f}(r) \widehat{g}(k)\sum\limits_{\substack{n \leq N}}c_r^{(s)}(n)c_k^{(s)}(n).
\end{align}
Now by Lemma \ref{lem2}, we get 
\begin{align*}
\sum\limits_{\substack{rk \leq U}}\widehat{f}(r) \widehat{g}(k)&\sum\limits_{\substack{n \leq N}}c_r^{(s)}(n)c_k^{(s)}(n)\\&= \sum\limits_{\substack{rk \leq U}}\widehat{f}(r) \widehat{g}(k)(\delta_{r,k} N c_r^{s}(0)+O(r^sk^s \log r^sk^s))\\
&=\sum\limits_{\substack{rk \leq U}}\widehat{f}(r) \widehat{g}(k)(\delta_{r,k} N \Phi_s(r^s)+O(r^sk^s \log r^sk^s))\\
&= N\sum\limits_{\substack{r^2 \leq U}}\widehat{f}(r) \widehat{g}(r)\Phi_s(r^s)+\sum\limits_{\substack{rk \leq U}}\widehat{f}(r) \widehat{g}(k)O(r^sk^s \log r^sk^s)).\\
\end{align*}
Now $\vert c_r^s(n)\vert \leq \sigma_{1,s}(n)$. Therefore $\sum\limits_{\substack{r,k}}\widehat{f}(r)\widehat{g}(k)$ is absolutely convergent. Since $rk < U,$ the error term is $O(U^s \log U^s)$. Thus we get
\begin{align*}
\sum\limits_{\substack{rk \leq U}}\widehat{f}(r) \widehat{g}(k)\sum\limits_{\substack{n \leq N}}c_r^{(s)}(n)c_k^{(s)}(n)&= N \sum\limits_{\substack{r^2 \leq U}}\widehat{f}(r) \widehat{g}(r)\Phi_s(r^s)+O(U^s \log U^s).
\end{align*}
By Lemma \ref{lem1}, the second sum in (\ref{equ-4}) is 
\begin{align*}
\sum\limits_{\substack{rk > U}}\widehat{f}(r) \widehat{g}(k)\sum\limits_{\substack{n \leq N}}c_r^{(s)}(n)c_k^{(s)}(n)
& \leq N \sum\limits_{\substack{rk > U}}\vert \widehat{f}(r)\vert  \vert\widehat{g}(k)\vert  \tau_s(r^s) \tau_s(k^s) (r^s,k^s)_s.\\
\end{align*}
Since $\sum\limits_{\substack{r,k}}\vert \widehat{f}(r)\vert  \vert\widehat{g}(k)\vert  \tau_s(r^s) \tau_s(k^s) (r^s,k^s)_s  < \infty$, we have $\sum\limits_{\substack{rk > U}}\widehat{f}(r) \widehat{g}(k)\sum\limits_{\substack{n \leq N}}c_r^{(s)}(n)c_k^{(s)}(n) << N$. 
Therefore $\sum\limits_{\substack{rk > U}}\widehat{f}(r) \widehat{g}(k)\sum\limits_{\substack{n \leq N}}c_r^{(s)}(n)c_k^{(s)}(n) = O(N)$.

Thus (\ref{equ-4}) becomes
\begin{align*}
 \sum\limits_{\substack{n \leq N}} f(n)g(n) & = N \sum\limits_{\substack{r^2 > U}}\widehat{f}(r) \widehat{g}(r) \Phi_s(r^s)+O(U^s \log U^s)+O(N)
 \\& = N \sum\limits_{\substack{r^2 > U}}\widehat{f}(r) \widehat{g}(r) \Phi_s(r^s)+O(U^s \log U^s)
 \\&= N \sum\limits_{\substack{r=1}}^{\infty}\widehat{f}(r) \widehat{g}(r) \Phi_s(r^s) - N \sum\limits_{\substack{r^2 \leq U}}\widehat{f}(r) \widehat{g}(r) \Phi_s(r^s) + O(U^s \log U^s).
\end{align*}
Since $\sum\limits_{\substack{r,k}}\vert \widehat{f}(r)\vert  \vert\widehat{g}(k)\vert  \tau_s(r^s) \tau_s(k^s) (r^s,k^s)_s  < \infty$, $\sum\limits_{\substack{r}}\vert \widehat{f}(r)\vert  \vert\widehat{g}(r)\vert  \tau_s(r^s)^2 r^s  < \infty$. \\
 Now  $\vert \widehat{f}(r)\vert \vert\widehat{g}(r)\vert \Phi_s(r^s) \leq \vert \widehat{f}(r)\vert \vert\widehat{g}(r)\vert r^s < \vert \widehat{f}(r)\vert \vert\widehat{g}(r)\vert r^s \tau_s(r^s)^2$, by comparison test, $\sum\limits_{\substack{r}}\widehat{f}(r) \widehat{g}(r) \Phi_s(r^s)$ converges absolutely.
 So $N \sum\limits_{\substack{r^2<U}}\widehat{f}(r) \widehat{g}(r) \Phi_s(r^s)<<N$. That is $N \sum\limits_{\substack{r^2<U}}\widehat{f}(r) \widehat{g}(r) \Phi_s(r^s)= O(N)$.

Thus,
 \begin{align*}
 \sum\limits_{\substack{n \leq N}} f(n)g(n) & = N \sum\limits_{\substack{r=1}}^{\infty}\widehat{f}(r) \widehat{g}(r) \Phi_s(r^s)+ O(U^s \log U^s)+O(N).  
\end{align*} If we choose $U$ to be of size $\frac{N^{\frac{1}{s}}}{\log N}$, then we have
 $\sum\limits_{\substack{n \leq N}} f(n)g(n)\sim N \sum\limits_{\substack{r}} \widehat{f}(r)\widehat{g}(r) \Phi_s(r^s)$.  
\end{proof}
We proceed to prove our second main result.

\begin{proof}[Proof of Theorem \ref{con-sum}]
We begin with the sum
\begin{align*}
\sum\limits_{\substack{n \leq N}} f(n)g(n+h)&= \sum\limits_{\substack{n \leq N}} \sum\limits_{\substack{r}}\widehat{f}(r)c_r^{(s)}(n)\sum\limits_{\substack{k}}\widehat{g}(k)c_k^{(s)}(n+h)
\\& = \sum\limits_{\substack{r,k}}\widehat{f}(r) \widehat{g}(k)\sum\limits_{\substack{n \leq N}} c_r^{(s)}(n)c_k^{(s)}(n+h)
\\&= \sum\limits_{\substack{rk \leq U}}\widehat{f}(r) \widehat{g}(k)\sum\limits_{\substack{n \leq N}} c_r^{(s)}(n)c_k^{(s)}(n+h)+\\& \sum\limits_{\substack{rk > U}}\widehat{f}(r) \widehat{g}(k)\sum\limits_{\substack{n \leq N}} c_r^{(s)}(n)c_k^{(s)}(n+h).
\end{align*}
First sum is
\begin{align*}
\sum\limits_{\substack{rk \leq U}}\widehat{f}(r) \widehat{g}(k)&\sum\limits_{\substack{n \leq N}} c_r^{(s)}(n)c_k^{(s)}(n+h)\\& = \sum\limits_{\substack{rk \leq U}}\widehat{f}(r) \widehat{g}(k)\{ \delta_{r,k}N c_r^s(h) + O(r^sk^s \log r^sk^s)\}\hspace{.2cm} \text{(by Lemma }\ref{lem2})
\\& = N \sum\limits_{\substack{r^2 \leq U}}\widehat{f}(r) \widehat{g}(r)c_r^s(h)+ \sum\limits_{\substack{rk \leq U}}\widehat{f}(r) \widehat{g}(k)O(r^sk^s \log r^sk^s)
\\&= N  \sum\limits_{\substack{r^2 \leq U}}\widehat{f}(r) \widehat{g}(r)c_r^s(h)+O(U^s \log U^s).
\end{align*}
Now we consider the absolute value of the second term.
\begin{align*}
\vert \sum\limits_{\substack{rk > U}}\widehat{f}(r) \widehat{g}(k)&\sum\limits_{\substack{n \leq N}} c_r^{(s)}(n)c_k^{(s)}(n+h) \vert \\&\leq
\sum\limits_{\substack{rk > U}}\vert\widehat{f}(r)\vert  \vert\widehat{g}(k)\left| \sum\limits_{\substack{n \leq N}} c_r^{(s)}(n)c_k^{(s)}(n+h) \right| \\& \leq \sum\limits_{\substack{rk > U}}\vert\widehat{f}(r)\vert  \vert\widehat{g}(k)\vert N^{\frac{1}{2}}(N+\vert h \vert)^{\frac{1}{2}} (r^sk^s)^{\frac{1}{2}} \tau_s(r^s) \tau_s(k^s) \text{(       by Lemma \ref{lem3})}
\\&\leq N^{\frac{1}{2}}(N+h)^{\frac{1}{2}} \sum\limits_{\substack{rk > U}}\vert\widehat{f}(r)\vert  \vert\widehat{g}(k)\vert (r^sk^s)^{\frac{1}{2}} \tau_s(r^s) \tau_s(k^s).
\end{align*}
Given that $\sum\limits_{\substack{rk>U}}\vert\widehat{f}(r)\vert  \vert\widehat{g}(k)\vert (r^sk^s)^{\frac{1}{2}} \tau_s(r^s) \tau_s(k^s) < \infty$.  Thus we get $\sum\limits_{\substack{rk > U}}\widehat{f}(r) \widehat{g}(k)\sum\limits_{\substack{n \leq N}} c_r^{(s)}(n)c_k^{(s)}(n+h) = O(N)$. Thus 
\begin{align*}
\sum\limits_{\substack{n \leq N}} f(n)g(n+h)&=N  \sum\limits_{\substack{r^2 \leq U}}\widehat{f}(r) \widehat{g}(r)c_r^s(h)+O(U^s \log U^s)+O(N)
\\& = N  \sum\limits_{\substack{r^2 \leq U}}\widehat{f}(r) \widehat{g}(r)c_r^s(h)+O(U^s \log U^s)
\\& =  N  \sum\limits_{\substack{r=1 }}^{\infty}\widehat{f}(r) \widehat{g}(r)c_r^s(h)- N  \sum\limits_{\substack{r^2 > U}}\widehat{f}(r) \widehat{g}(r)c_r^s(h)+O(U^s \log U^s).
\end{align*}
Now we have  $\sum\limits_{\substack{r,k}}\vert\widehat{f}(r)\vert  \vert\widehat{g}(k)\vert (r^sk^s)^{\frac{1}{2}} \tau_s(r^s) \tau_s(k^s) < \infty$. Therefore  $\sum\limits_{\substack{r}}\vert\widehat{f}(r)\vert  \vert\widehat{g}(r)\vert r^s \tau_s(r^s)^2 < \infty$. Now 
\begin{align*}
\vert \widehat{f}(r) \widehat{g}(r)c_r^s(h) \vert & \leq \vert\widehat{f}(r) \widehat{g}(r) r^s \tau_s(r^s) \vert 
\\& \leq \vert\widehat{f}(r) \widehat{g}(r) r^s \tau_s(r^s)^2 \vert .
\end{align*}
 By comparison test, $\sum\limits_{\substack{r}}\widehat{f}(r) \widehat{g}(k) c_r^s(h)$ converges absolutely. 
 \begin{align*}
 \sum\limits_{\substack{n \leq N}} f(n)g(n+h)&=  N  \sum\limits_{\substack{r=1 }}^{\infty}\widehat{f}(r) \widehat{g}(r)c_r^s(h)+O(U^s \log U^s) +O(N)
 \\&=  N  \sum\limits_{\substack{r=1 }}^{\infty}\widehat{f}(r) \widehat{g}(r)c_r^s(h)+O(U^s \log U^s).
 \end{align*}
 Thus we get $\sum\limits_{\substack{n \leq N}} f(n)g(n+h) \sim N  \sum\limits_{\substack{r=1 }}^{\infty}\widehat{f}(r) \widehat{g}(r)c_r^s(h))$.
\end{proof}

In \cite{Chandran2022OnAR} we proved that, for $k, s\geq 1$,$\frac{\sigma_{ks}(n)}{n^{ks}} =\zeta(k+1) \sum\limits_{\substack{{r=1}}}^{\infty}\frac{c_r^{(s)}(n^s)}{r^{(k+1)s}}$, where $\sigma_s(n)=\sum\limits_{d|n}d^s$. Now we have the following corollary of Theorem \ref{con-sum}, which is similar to [Corollary 1, \cite{murty2015error}].

\begin{coro}
For $a,b >1+1/2$, $s\geq 1$ and $h=m^s k$, where $k$ is an $s^{th}$ power free integer, we have
\begin{align*}
\sum\limits_{n\leq N} \frac{\sigma_{as}(n)}{n^{as}}\frac{\sigma_{bs}(n+h)}{(n+h)^{bs}} \sim N \frac{\zeta(a+1)\zeta(b+1)}{\zeta(a+b+2)}\sigma_{-(a+b+1)s}(m).
\end{align*} 

\end{coro}
\begin{proof}
We first consider the sum,
\begin{align*}
 \sum\limits_{\substack{r,k}}\vert \widehat{f}(r)\vert \vert\widehat{g}(k) \vert (r^s,k^s)^{\frac{1}{2}} \tau_s(r^s) \tau_s(k^s) & \leq \sum\limits_{r,k} \vert \frac{\zeta(a+1)}{r^{(a+1)s}} \vert \vert \frac{\zeta(b+1)}{k^{(b+1)s}} \vert (r^sk^s)^{\frac{1}{2}}r^sk^s
 \\&=   \vert \zeta(a+1) \vert \vert \zeta(b+1)\vert  \sum\limits_{r,k} \frac{1}{(r^s)^{a-\frac{1}{2}}} \frac{1}{(k^s)^{b-\frac{1}{2}}}
 \\& < \infty.
\end{align*}
Now by Theorem \ref{con-sum}, 
\begin{align*}
\sum\limits_{n\leq N} \frac{\sigma_{as}(n)}{n^{as}}&\frac{\sigma_{bs}(n+h)}{(n+h)^{bs}}  \\&\sim N  \sum\limits_{r}\frac{\zeta(a+1)}{r^{(a+1)s}}\frac{\zeta(b+1)}{r^{(b+1)s}} c_r^s(h)
\\&= N \zeta(a+1)\zeta(b+1)\sum\limits_{r}\frac{c_r^s(h)}{r^{(a+b+2)s}}
\\&= N \zeta(a+1)\zeta(b+1)\sum\limits_{r}\frac{c_r^s(m^s)}{r^{(a+b+2)s}}\hspace{10pt} (\text{since }c_r^s(m^sk)=c_r(m^s))
\\&= N \zeta(a+1)\zeta(b+1)\frac{\sigma_{(a+b+1)s}(m)}{m^{(a+b+1)s}
}\frac{1}{\zeta(a+b+2)}
\\& = N \frac{\zeta(a+1)\zeta(b+1)}{\zeta(a+b+2)} \frac{\sum\limits_{d \mid m}d^{(a+b+1)s}}{m^{(a+b+1)s}}
\\& =  N \frac{\zeta(a+1)\zeta(b+1)}{\zeta(a+b+2)} \sum\limits_{k \mid m}k^{-(a+b+1)s} \hspace{1cm} (\text{where }m =kd)
\\&=  N \frac{\zeta(a+1)\zeta(b+1)}{\zeta(a+b+2)}\sigma_{-(a+b+1)s}(m).
\end{align*} 
\end{proof}
 We derive a Cohen-Ramanujan expansion which appears similar to the expansion (\ref{Jordan_sum}) in the next theorem. The purpose of providing this expansion is to give the example following this theorem to illustrate Theorem \ref{con-sum}.
 
 \begin{theo}\label{coh_sum_J_n}
 For $s,k \geq 1$ and $n \in \mathbb{N}$, we have 
 \begin{align*}
 \frac{J_k(n)}{n^k} = \sum\limits_{q} \frac{\mu(q)}{\zeta(s+k)J_{s+k}(q)}c_q^s(n^s).
 \end{align*}
 \end{theo}
\begin{proof}
If $s,k\geq 1$, it is easy to see that $\sum\limits_{q} \frac{\mu(q)c_q^s(n)}{J_k(q)} = \prod\limits_{\substack{p\\p \text{ prime}}}(1+\frac{\mu(p)c_p^s(n)}{J_k(p)})$. So,
\begin{align*}
\sum\limits_{q} \frac{\mu(q)c_q^s(n^s)}{J_{s+k}(q)}&=\prod\limits_{\substack{p\\p \text{ prime}}}(1+\frac{\mu(p)c_p^s(n^s)}{J_{s+k}(p)})
\\&= \prod\limits_{\substack{p\mid n \\p \text{ prime}}}(1+\frac{\mu(p)c_p^s(n^s)}{J_{s+k}(p)})\prod\limits_{\substack{p\nmid n\\p \text{ prime}}}(1+\frac{\mu(p)c_p^s(n^s)}{J_{s+k}(p)})
\\&=  \prod\limits_{\substack{p\mid n \\p \text{ prime}}}(1-\frac{p^s-1}{p^{s+k}-1})\prod\limits_{\substack{p\nmid n\\p \text{ prime}}}(1+\frac{1}{p^{s+k}-1})
\\&= \prod\limits_{\substack{p}}(\frac{1}{1-p^{-(s+k)}})\prod\limits_{\substack{p\mid n\\p \text{ prime}}}(1-\frac{1}{p^k})
\\& = \zeta(k+s) \frac{J_k(n)}{n^k}  \hspace{1cm}\text{(by Lemma \ref{prod_zeta})}.
\end{align*}
\end{proof}
 The next result gives another example to demonstrate Theorem  \ref{con-sum}, which is similar to [Corollary 2, \cite{murty2015error}].
\begin{coro}
If $s>1$, $a,b > 1+\frac{s}{2}$ and $h= m^sk$, where $k$ is $s-$ power free integer, then
\begin{align*}
\sum\limits_{n \leq N} \frac{J_a(n)}{n^a} \frac{J_b(n+h)}{(n+h)^b}&\sim N \prod\limits_{p \mid m}\Big(\big(1-\frac{1}{p^{s+a}}\big)\big(1-\frac{1}{p^{s+b}}\big)+\frac{p^{s}-1}{p^{a+b+2s}}\Big)\\&\times\prod\limits_{p \nmid m}\Big(\big(1-\frac{1}{p^{s+a}}\big)\big(1-\frac{1}{p^{s+b}}\big)-\frac{1}{p^{a+b+2s}}\Big).
\end{align*}
\end{coro}
\begin{proof}
We have from equation (\ref{Jordan_sum}),
\begin{align*}
&\frac{J_a(n)}{n^a} = \sum\limits_{r} \frac{\mu(r)}{J_{s+a}(r)\zeta(s+a)}c_r^s(n^s) = \sum\limits_{r} \widehat{f}(r) c_r^s(n^s) \text{ and }
\\& \frac{J_b(n)}{n^b} = \sum\limits_{t} \frac{\mu(t)}{J_{s+b}(t)\zeta(s+b)}c_t^s(n^s) = \sum\limits_{t} \widehat{g}(t) c_t^s(n^s) .
\end{align*}
Now we consider the convergence of the sum 
\begin{align*}
\sum\limits_{\substack{r,t}}\vert \widehat{f}(r)\vert \vert\widehat{g}(t)& \vert (r^s,t^s)^{\frac{1}{2}} \tau_s(r^s) \tau_s(t^s)\\ & \leq \sum\limits_{\substack{r,t}} \frac{1}{J_{s+a}(r)\zeta(s+a)} \frac{1}{J_{s+b}(t)\zeta(s+b)}(r^st^s)^{\frac{1}{2}}r^st^s.
\end{align*}

Since $\frac{1}{J_s(r)}<<\frac{1}{r^s}$, and $a,b > 1+\frac{s}{2}$, the above sum converges. Therefore by Theorem \ref{con-sum}, we have 
\begin{align*}
\sum\limits_{n \leq N} &\frac{J_a(n)}{n^a} \frac{J_b(n+h)}{(n+h)^b}\\&\sim N \sum\limits_{r} \frac{\mu(r)}{J_{s+a}(r)\zeta(s+a)}\frac{\mu(r)}{J_{s+b}(r)\zeta(s+b)}c_r^s(h)
\\&=  N \sum\limits_{r} \frac{\mu(r)}{J_{s+a}(r)\zeta(s+a)}\frac{\mu(r)}{J_{s+b}(r)\zeta(s+b)}c_r^s(m^s)\hspace{10pt}(\text{since }c_r^s(m^sk)=c_r^s(m^s))
\\&= \frac{N}{\zeta(s+a)\zeta(s+b)} \sum\limits_{r} \frac{\vert \mu(r) \vert}{J_{s+a}(r)J_{s+b}(r)} c_r^s(m^s)
\\&= \frac{N}{\zeta(s+a)\zeta(s+b)} \Big(\frac{\vert \mu(1) \vert}{J_{s+a}(1)J_{s+b}(1)} c_1^s(m^s)+\sum\limits_{p_i}\frac{\vert \mu(p_i) \vert}{J_{s+a}(p_i)J_{s+b}(p_i)} c_{p_i}^s(m^s)\\&+ \hdots + \sum\limits_{p_ip_j}\frac{\vert \mu(p_ip_j) \vert}{J_{s+a}(p_ip_j)J_{s+b}(p_ip_j)} c_{p_ip_j}^s(m^s)
\\&= \frac{N}{\zeta(s+a)\zeta(s+b)} \Big(1+\sum\limits_{p_i}\frac{c_{p_i}^s(m^s)}{(p_i^{s+a}-1)(p_i^{s+b}-1)} \\&+ \sum\limits_{p_ip_j}\frac{c_{p_ip_j}^s(m^s)}{(p_i^{s+a}-1)(p_i^{s+b}-1)(p_j^{s+a}-1)(p_j^{s+b}-1)}+\hdots\Big)
\\&= \frac{N}{\zeta(s+a)\zeta(s+b)}\Big( \big(1+\frac{c_{p_1}^s(m^s)}{(p_1^{s+a}-1)(p_1^{s+b}-1)}\big)\big(1+\frac{c_{p_2}^s(m^s)}{(p_2^{s+a}-1)(p_2^{s+b}-1)}\big)\hdots \Big)
\\&= \frac{N}{\zeta(s+a)\zeta(s+b)} \prod\limits_{p}\big(1+\frac{c_{p}^s(m^s)}{(p^{s+a}-1)(p^{s+b}-1)}\big)
\\&= N \frac{1}{\prod\limits_{p}(\frac{1}{1-p^{-(s+a)}})\prod\limits_{p}(\frac{1}{1-p^{-(s+b)}})}\big(1+\frac{c_{p}^s(m^s)}{(p^{s+a}-1)(p^{s+b}-1)}\big)
\\&= N \prod\limits_{p}\Big((1-\frac{1}{p^{s+a}})(1-\frac{1}{p^{s+b}})\big(1+\frac{c_{p}^s(m^s)}{(p^{s+a}-1)(p^{s+b}-1)}\big) \Big)
\\&= N \prod\limits_{p}\Big((1-\frac{1}{p^{s+a}})(1-\frac{1}{p^{s+b}})+\frac{c_{p}^s(m^s)}{p^{a+b+2s}} \Big)
\\& =N \prod\limits_{p\mid m}\Big((1-\frac{1}{p^{s+a}})(1-\frac{1}{p^{s+b}})+\frac{p^s-1}{p^{a+b+2s}} \Big)\prod\limits_{p\nmid m}\Big((1-\frac{1}{p^{s+a}})(1-\frac{1}{p^{s+b}})-\frac{1}{p^{a+b+2s}} \Big).
\end{align*}
\end{proof}  
  
Now we proceed to discuss the existence of Cohen-Ramanujan expansion for $f_h$ derived from $f$.
\begin{lemm}\label{lem4}
For $h \leq N$, $\sum\limits_{\substack{n \leq N}}c_r^s(n)c_k^s(n+h) \leq 2N\Phi_s(r^s)\tau(k)$.
\end{lemm}
\begin{proof}
Cosider the sum 
\begin{align*}
\sum\limits_{\substack{n \leq N}}c_r^s(n)c_k^s(n+h) & = \sum\limits_{\substack{n \leq N}}\sum\limits_{\substack{a = 1\\(a,r^s)_s=1}}^{r^s} e^{\frac{2 \pi i an}{r^s}}\sum\limits_{\substack{d \mid k\\d^s\mid n+h}} \mu(\frac{k}{d}) d^s
\\& = \sum\limits_{\substack{a = 1\\(a,r^s)_s=1}}^{r^s} e^{\frac{2 \pi i an}{r^s}} \sum\limits_{\substack{d \mid k}} \mu(\frac{k}{d}) d^s \sum\limits_{\substack{n \leq N \\ d^s \mid n+h}} 1
\\& = \sum\limits_{\substack{a = 1\\(a,r^s)_s=1}}^{r^s} e^{\frac{2 \pi i an}{r^s}} \sum\limits_{\substack{d \mid k}} \mu(\frac{k}{d}) d^s \sum\limits_{\substack{n+h \leq N+h \\ d^s \mid n+h}} 1
\\& = \sum\limits_{\substack{a = 1\\(a,r^s)_s=1}}^{r^s} e^{\frac{2 \pi i an}{r^s}} \sum\limits_{\substack{d \mid k}} \mu(\frac{k}{d}) d^s \left[\left[\frac{N+h}{d^s}\right]\right]
\\& \leq \sum\limits_{\substack{a = 1\\(a,r^s)_s=1}}^{r^s} e^{\frac{2 \pi i an}{r^s}} \sum\limits_{\substack{d \mid k}} \mu(\frac{k}{d}) d^s \frac{N+h}{d^s}
\\ & \leq (N+h) \sum\limits_{\substack{a = 1\\(a,r^s)_s=1}}^{r^s} 1 \sum\limits_{\substack{d \mid k}} 1
\\ & \leq (N+h) \Phi_s(r^s) \tau(k)
\\ & \leq 2N \Phi_s(r^s) \tau(k).
\end{align*}
\end{proof}
If $d \mid r$ and $t\mid r$, then by the orthogonality relation,
\begin{align*}
 \frac{1}{r^s}\sum\limits_{\substack{m=1}}^{r^s}c_d^s(m)c_t^s(m) = \begin{cases}
\Phi_s(d^s), \text{ if } d =t\\0, \text{ otherwise.} \end{cases}
\end{align*} This orthogonality relation allows us to write down the possible candidates for the Cohen-Ramanujan coefficients of any given arithmetical function. Indeed if $f(n)  =  \sum\limits_{\substack{q = 1}}^{\infty}\widehat{f}(q)c_q^s(n) \text{ then } f(n)c_r^s(n) = \sum\limits_{\substack{q = 1}}^{\infty}\widehat{f}(q)c_q^s(n)c_r^s(n).$
So we get,
\begin{align*}
\lim_{x^s \to \infty} \frac{1}{x^s} \sum\limits_{\substack{n \leq x^s}} f(n) c_r^s(n)&= \lim_{x^s \to \infty} \frac{1}{x^s} \sum\limits_{\substack{n \leq x^s}} \sum\limits_{\substack{q = 1}}^{\infty}\widehat{f}(q)c_q^s(n) c_r^s(n)
\\&= \sum\limits_{\substack{q = 1}}^{\infty}\widehat{f}(q) \lim_{x^s \to \infty} \frac{1}{x^s} \sum\limits_{\substack{n \leq x^s}} c_q^s(n) c_r^s(n)
\\&= \widehat{f}(r)  \lim_{x^s \to \infty} \Phi_s(r^s).\\
\text{ Therefore, } \widehat{f}(r)&= \frac{M(fc_r^s)}{\Phi_s(r^s)},
\end{align*}
where by $M(f)$ we mean the limit $\lim\limits_{\substack{x \to \infty}} \frac{1}{x} \sum\limits_{\substack{n \leq x}} f(n) $.

We are now ready to prove Theorem \ref{th:fhexpansion}.
\begin{proof}[Proof of Theorem \ref{th:fhexpansion}]
Suppose that $f(n) =  \sum\limits_{\substack{q = 1}}^{\infty}\widehat{f}(q)c_q^s(n)$ .  Therefore 
\begin{align*}
 \sum\limits_{\substack{n \leq N}}f(n+h)c_r^s(n) &= \sum\limits_{\substack{n \leq N}} \sum\limits_{\substack{q = 1}}^{\infty}\widehat{f}(q)c_q^s(n+h) c_r^s(n)
 \\&= \sum\limits_{\substack{q = 1}}^{\infty}\widehat{f}(q)\sum\limits_{\substack{n \leq N}}c_r^s(n)c_q^s(n+h)
 \\&= \sum\limits_{\substack{q \leq U}}\widehat{f}(q)\sum\limits_{\substack{n \leq N}}c_r^s(n)c_q^s(n+h)+ \sum\limits_{\substack{q > U}}\widehat{f}(q)\sum\limits_{\substack{n \leq N}}c_r^s(n)c_q^s(n+h).
\end{align*}
Now consider the first term in this summation.
\begin{align*}
\sum\limits_{\substack{q \leq U}}\widehat{f}(q)&\sum\limits_{\substack{n \leq N}} c_r^s(n)c_q^s(n+h) \\&=\sum\limits_{\substack{q \leq U}}\widehat{f}(q)\left[\delta_{r,q} N c_r^s(h)+O(r^sq^s\log(r^sq^s))\right] \hspace{.5cm}\text{(by Lemma }\ref{lem2})
\\&= Nc_r^s(h)+O(r^sU^s\log (r^sU^s)).
\end{align*}
Second sum is 
\begin{align*}
\sum\limits_{\substack{q > U}}\widehat{f}(q)\sum\limits_{\substack{n \leq N}}c_r^s(n)c_q^s(n+h)& \leq 2N \Phi_s(r^s) \sum\limits_{\substack{q > U}}\widehat{f}(q)\tau(q) \hspace{.5cm}({\text{by Lemma \ref{lem4}}}).
\end{align*}
Since $ \sum\limits_{\substack{q }} \vert \widehat{f}(q) \vert \tau(q) < \infty$, $\sum\limits_{\substack{q > U}}\widehat{f}(q)\sum\limits_{\substack{n \leq N}}c_r^s(n)c_q^s(n+h)=O(N).$
Therefore $ \sum\limits_{\substack{n \leq N}}f(n+h)c_r^s(n) = N\widehat{f}(r)c_r^s(h)+O(r^sU^s\log (r^sU^s))+O(N)$. Now
\begin{align*}
M(f_hc_r^s) &= \lim_{N \to \infty} \frac{1}{N} \sum\limits_{\substack{n \leq N}}f_h(n)c_r^s(n)
\\&=  \lim_{N \to \infty} \frac{1}{N} \sum\limits_{\substack{n \leq N}}f(n+h)c_r^s(n)
\\&=  \lim_{N \to \infty} \frac{1}{N} \left[N \widehat{f}(r)c_r^s(h)+O(r^sU^s\log (r^sU^s))+O(N)\right]
\\& = \widehat{f}(r)c_r^s(h).
\end{align*}
Now by orthogonality relation, we have 
\begin{align*}
\widehat{f}_h(r) &= \frac{M(f_hc_r^s)}{\Phi_s(r^s)}
\\& = \frac{\widehat{f}(r)c_r^s(h)}{\Phi_s(r^s)}.
\end{align*}
\end{proof}
 
\end{document}